\documentclass[12pt]{amsart}
\usepackage{epsfig,color}
\usepackage{amssymb} 

\makeatother

\theoremstyle{definition}

\def\fnum{equation}
\newtheorem{Thm}[\fnum]{Theorem}

\newtheorem{Lem}[\fnum]{Lemma}

\newtheorem{Def}[\fnum]{Definition}

\newtheorem{Pro}[\fnum]{Proposition}

\numberwithin{equation}{section}

\newcommand{\dist}{{\text {dist}}}

\def\RR{{\bold R}}

\def\SS{{\bold S}}

\newcommand{\e}{{\text {e}}}

\newcommand{\Length}{{\text {Length}}}

\newcommand{\cC}{{\mathcal{C}}}

\newcommand{\cO}{{\mathcal{O}}}

\newcommand{\eqr}[1]{(\ref{#1})}

\title{Quantitative uniqueness for mean curvature flow}
\author[]{Tobias Holck Colding}%
\address{MIT, Dept. of Math.\\
77 Massachusetts Avenue, Cambridge, MA 02139-4307.}
\author[]{William P. Minicozzi II}%
 
\thanks{The  authors
were partially supported by NSF  DMS Grants   2405393 and 2304684.}

\email{colding@math.mit.edu  and minicozz@math.mit.edu}

\begin{document}

\maketitle

{\centering\footnotesize To our friend Gang Tian.\par}

\begin{abstract}
We show how to use the arguments of \cite{CM2} to get a stronger effective version of uniqueness of blowups that has a number of consequences.  
\end{abstract}

\section{Introduction}

There is a natural   scaling for a mean curvature flow   (MCF)  $M_s \subset \RR^{n+1}$, where space and time dilate parabolically.
A general limit flow at a space-time point $(\bar{x} , \bar{s})$ is a limit of a sequence of rescalings  $\frac{1}{\mu_i} \, \left( M_{s_i - \mu_i^2 \, s} - x_i \right)$ centered at a sequence of points
$(x_i , s_i) \to  (\bar{x} , \bar{s})$ with $\mu_i \to 0$.   Typically, the time-slices of the limit are non-compact and the convergence is on compact sets.
When the dilations are all centered at the same point, then the limit flow is called a tangent flow.

  A tangent flow  at the origin in space-time is the limit of 
a sequence of rescaled flows $\frac{1}{\delta_i} \, M_{\delta_i^2 \, t}$ where $\delta_i \to 0$.
By a monotonicity formula of Huisken, \cite{H}, and an argument of Ilmanen and
White, \cite{I,W2}, tangent flows are   shrinkers, i.e., self-similar solutions of 
MCF that evolve by rescaling. A priori, different sequences $\delta_i$ could give different tangent flows
and the question of the uniqueness of the blowup is whether it is independent of the sequence.  
Uniqueness has strong implications for   regularity, \cite{CM4}, cf. \cite{W3,W4}.

A singular point is {\emph{cylindrical}} if some tangent flow is a multiplicity one cylinder $\SS^k \times \RR^{n-k}$; 
  \cite{CM2} proved that   cylindrical blowups are   unique.  The main tool was a  Lojasiewicz-type inequality that led to a rate of 
decay for the gaussian area and  a rate of convergence to the limit.  This was inspired by the way that  Lojasiewicz proved  uniqueness  for finite dimensional analytic gradient flows, \cite{L}.   The formal similarities  to  \cite{L} helped frame the problem.  However,  those methods did not  apply
for a number of reasons, including the non-compactness of the time-slices. Instead  new ideas and techniques were  required.

\vskip1mm
The next theorem will illustrate why  effective uniqueness is useful.
  Suppose that for $s \in [-1,1]$, 
    $\lambda (M_{s}) \leq \lambda_0$ and $M_s$ satisfies:
\begin{enumerate}
\item[(A)] The origin $(0,0)$ is a cylindrical singularity with blow up $\cC=\SS^k_{\sqrt{2k}} \times \RR^{n-k}$.
\item[(B)] There are sequences $x_i \to 0$, $s_i \to 0$ and $\mu_i \to 0$ so that 
\begin{align}
	\frac{1}{\mu_i} \, \left( M_{s_i - \mu_i^2} - x_i \right) \notag
\end{align}
converges smoothly with multiplicity one to   $\cO(\cC)$, where $\cO$ is a rotation in $\RR^{n+1}$.
\end{enumerate}
We will see that $\cO (\cC) = \cC$, i.e., the two cylinders are the same:

\begin{Thm}	\label{t:AB}
If (A) and (B) hold, then $\cO (\cC) = \cC$.
\end{Thm}

If the sequence in (B) was centered at the origin in space-time (i.e., if $x_i =0$ and $s_i = 0$), then this is uniqueness of 
cylindrical blow ups from \cite{CM2}.  The more general case, where the centers of the rescalings    converge to the origin, will follow from 
an effective uniqueness theorem; see Theorem \ref{t:close} below.   

\vskip1mm
It is cleanest to state the effective uniqueness for solutions of the rescaled MCF.   The rescaled MCF is the gradient flow for the $F$-functional 
or Gaussian surface
area 
\begin{align}
  F (\Sigma) = (4\pi)^{-n/2} \, \int_{\Sigma} \,
  \e^{-\frac{|x|^2}{4}} \, d\mu \, .
\end{align}
The entropy, \cite{CM1}, is the supremum of the Gaussian surface areas
 over all centers and scales
  \begin{align}
  	\lambda (\Sigma) =\sup_{c >0 , x_0\in \RR^{n+1}} F (x_0 + c\, \Sigma) \, .
\end{align}

\begin{Def}
We will say that $\dist_R (\Sigma , \Gamma) < \epsilon$ if 
  $B_R \cap \Sigma$ can be written as a $C^{2,\alpha}$ graph over (a subset of) $\Gamma$ of a function with $C^{2,\alpha}$ norm less than $\epsilon$.
\end{Def}

The next theorem shows that if a rescaled MCF starts off close to a cylinder and  $F$  does not decrease much,  then the flow  does not change much. 
They key point is that this is independent of the time flowed.  In the theorem, $\Sigma_t$ is an $n$-dimensional rescaled MCF  with entropy $\lambda (\Sigma_{t}) \leq \lambda_0$ and all constants  are allowed to depend on $n$ and $\lambda_0$.

\begin{Thm}	\label{t:close}
There exist $c, \alpha, \epsilon_1 , \epsilon_2> 0$ and  $R_1 , R_2>2n$ so that if $\Sigma_t$ is defined on   $ [t_1 , t_2]$,
\begin{enumerate}
\item $\dist_{R_1} ( \Sigma_t, \cC) < \epsilon_1 $ for $t \in [t_1 , t_1+2] $, and 
\item $|F(\Sigma_{t_i}) - F(\cC)|  < \epsilon_2$ for $i=1,2$, 
\end{enumerate}
then  $\dist_{R_2} (\Sigma_t ,\Sigma_{t_1+1}) < c \, |F(\Sigma_{t_1}) - F(\cC)|^{\alpha} + c \, |F(\Sigma_{t_2}) - F(\cC)|^{\alpha}$ for $t\in [t_1+1 , t_2]$.
\end{Thm}

\vskip1mm
 The proof of Theorem \ref{t:close}  gives a   stronger notion of closeness and a bound for the distance traveled, but  the statement of Theorem \ref{t:close} 
 suffices for  applications.  We will use the theorem  to prove  Theorem \ref{t:AB}. It also implies the uniqueness of cylindrical blow down  limits for ancient mean curvature flows.

\vskip2mm
We are grateful to Brian White for raising the question that is proven in Theorem \ref{t:AB}.  

\section{Model case}

We will illustrate the ideas behind the effective uniqueness results   in the  finite dimensional model case of the gradient flow of an analytic function. 

Suppose that $F$ is a function on $\RR^n$ with $\nabla F (0) = 0$ and $F$ satisfies the gradient Lojasiewicz inequality, \cite{L}, 
\begin{align}	\label{e:gL}
	\left| F(x) -F(0) \right|^{1 + \tau} \leq |\nabla F(x)|^2 
\end{align}
for some $\tau \in (1/3, 1)$ and all $x \in B_2$.

Let  $\gamma (t)$ be  a gradient flow line for $F$, so that $\gamma'(t) = - \nabla F\circ \gamma (t)$ and 
\begin{align}	\label{e:gLhere}
	\partial_t \, F(\gamma(t)) = - |\nabla F|^2 (\gamma(t)) \, .
\end{align}
Lojasiewicz used \eqr{e:gL} to prove that if $\gamma$ has $0$ as a limit point, then $\gamma$ has finite length and converges to $0$ as $t \to \infty$; this is 
the Lojasiewicz uniqueness theorem, \cite{L}.

The next proposition shows that if a subsegment of $\gamma$ starts and ends near the critical point, then it's  length has a fixed upper bound that is independent of the time that 
$\gamma$ flows.  The Lojasiewicz uniqueness theorem is a corollary, but the proposition applies   more generally, including    where $F \circ \gamma$ goes just a bit below $F(0)$.

\begin{Pro}	\label{p:model}
There exist  $c , \alpha , \epsilon > 0$ so that if $\gamma$ is defined on $[t_1,t_2]$,  
\begin{align}
	\gamma (t_i) \subset B_{\frac{1}{4}} {\text{ and }}  |F(0) - F(\gamma(t_i))| < \epsilon \, , 
\end{align}
then $\gamma$   has length at most $c\, |F(\gamma(t_1)) - F(0)|^{\alpha}  + c\, | F(0) - F(\gamma(t_2))|^{\alpha} < \frac{1}{2}$. 
\end{Pro}

\vskip2mm
 Proposition \ref{p:model} is stated as a length bound since that is easy to understand and  clearly implies the uniqueness theorem.  The argument 
 actually gives the slightly stronger bound
 \begin{align}	\label{e:slightly}
 	\sum_i \left[ F(\gamma (i)) - F(\gamma(i+1)) \right]^{\frac{1}{2}} \leq c\, |F(\gamma(t_1)) - F(0)|^{\alpha}  + c\, | F(0) - F(\gamma(t_2))|^{\alpha}  \, .
 \end{align}
 The reason this is stronger is because the fundamental theorem of calculus and the Cauchy-Schwarz inequality give that
 \begin{align}	\label{e:Lgamma}
 	\left( \int_i^{i+1} |\gamma_t| \right)^2 &= \left( \int_i^{i+1}  \sqrt{ - \partial_t \, (F  \circ \gamma)} \right)^2 \leq - \int_i^{i+1}  \partial_t \, (F  \circ \gamma) \notag \\
	&=
	F(\gamma (i)) - F(\gamma(i+1) \, ,  
 \end{align}
 so the length of the segment of $\gamma$ from $t=i$ to $t=i+1$ is at most $\left( F(\gamma (i)) - F(\gamma(i+1) \right)^{ \frac{1}{2}}$.

\subsection{The proof in the model case}

The next lemma proves the proposition in the special case where $F\circ \gamma$ stays above $F(0)$.  

\begin{Lem}	\label{l:above}
There exist  $c , \alpha , \epsilon > 0$ so that if $\gamma$ is defined on $[t_1,t_2]$, $\gamma(t_1) \in B_{\frac{1}{2}}$,  and 
\begin{align}
	F(0) < F(\gamma(t_2))  < F(\gamma(t_1)) < F(0) + \epsilon \, , 
\end{align}
then $\gamma$   has length at most $c\, [ F(\gamma(t_1)) - F(0)]^{\alpha} < \frac{1}{4}$. 
\end{Lem}

\begin{proof}
It is convenient to translate in time so that $t_1 = 0$.  Define $T \in (0 , t_2]$ by
\begin{align}
	T = \max \, \{ t \in (0, t_2] \, | \, \gamma (t) \in \overline{B_1} \} \, .
\end{align}
Given $t \in [0,T]$, 
set $f(t) = F(\gamma(t)) - F(0)$, so that $0 <f < \epsilon$, $f' = - |\nabla F|^2 (\gamma(t))$ and \eqr{e:gL} gives that
\begin{align}	\label{e:fromgL}
	f^{1+\tau} \leq - f' \, .
\end{align}
Since $f> 0$, $f^{-\tau}$ is well-defined and the chain rule gives 
\begin{align}
	\left( f^{-\tau} \right)' = - \tau \, f^{-1-\tau} \, f' \geq \tau \, .
\end{align}
Integrating this gives that $f^{-\tau} (t) \geq f^{-\tau} (0) + \tau \, t$ and, thus, 
\begin{align}	\label{e:ftau}
	f(t) \leq \left( f^{-\tau}(0) + \tau \, t \right)^{ - \frac{1}{\tau}} \, .
\end{align}
Since $\tau < 1$, we have that $\frac{1}{\tau} = 1 + 3 \, \delta$ with $\delta > 0$.  Let $L$ be the length of $\gamma$ restricted to $[0,T]$.  Since $\gamma$ is a gradient flow, using the Cauchy-Schwarz inequality gives
\begin{align}	\label{e:fromcs}
	L^2 &=    \left( \int_0^T \sqrt{-f'} \right)^2 \leq \left( \int_0^T (-f') \, (1+t)^{1+\delta} \right) \, \int_0^T (1+t)^{-1-\delta} \notag \\
	&\leq
	\frac{1}{\delta} \, \left( \int_0^T (-f') \, (1+t)^{1+\delta} \right) \, .
\end{align}
 Integrating by parts and using \eqr{e:ftau}  bounds the last integral by 
\begin{align}
	\int_0^T (-f') \, (1+t)^{1+\delta}  &= - [ f\, (1+t)^{1+\delta}]_0^T + (1+\delta) \,  \int_0^T f \, (1+t)^{\delta} \leq f (0) + 2\, \int_0^T f \, (1+t)^{\delta}  \notag \\
		&\leq f(0) + 2 \, \int_0^T  \left( f^{-\tau}(1) + \tau \, t \right)^{  -1 - 3\, \delta} \, (1+t)^{\delta}  \\
		&\leq f(0) + 2 \, (f(0))^{\delta \, \tau} \,  \int_0^{T}  \, \left(1 + \tau \, t \right)^{  -1 - 2\, \delta} \, (1+t)^{\delta} \leq c' \,[ f(0)]^{2\, \alpha} 
		\notag \, ,
\end{align}
where $c' , \alpha  > 0$ depend on $\delta$ but are independent of $T$.  Using this in \eqr{e:fromcs} gives that
\begin{align}
	L \leq c \, [f(0)]^{\alpha} \leq c \, \epsilon^{\alpha}\, .	\label{e:modelinside}
\end{align}
As long as $\epsilon > 0$ is small enough, then $\gamma$ must stay inside $B_{  \frac{3}{4}} \subset B_1$, so we conclude that $T = t_2$ and, thus, $L$ is the entire length of $\gamma$ on $[t_1 , t_2]$.
\end{proof}

The previous lemma is all that would be needed for the Lojasiewicz uniqueness theorem.  We turn next to the general case where $F \circ \gamma$ is allowed to go below $F(0)$.

\begin{proof}[Proof of Proposition \ref{p:model}]
Let $L = \Length (\gamma)$. We will consider three cases, depending on whether $F$ is above or below $F(0)$ along $\gamma$.   

\vskip1mm
\noindent
{\bf{Case 1}}: $F(\gamma(t_2) ) \geq F(0)$.  By continuity, we can assume that $F(\gamma(t_2) ) > F(0)$ and, thus, Lemma \ref{l:above} gives that
\begin{align}
	L \leq c\, [ F(\gamma(t_1)) - F(0)]^{\alpha}  \, .
\end{align}

\vskip1mm
\noindent
{\bf{Case 2}}: $F(0) \geq   F(\gamma(t_1) )$. Similarly to the first case, we can assume that $F(0) > F(\gamma(t_1))$.  In this case, we  reverse the 
parameterization{\footnote{The assumptions here are symmetric in $t_1$ and $t_2$, so case 2 can be reduced to case 1.  MCF is not symmetric under time reversal
and we will use a slightly different approach there.}}
  of $\gamma$ to get a curve $\tilde{\gamma} (t) = \gamma (-t)$ that is the gradient flow to $\tilde{F} = - F$.  The curve $\tilde{\gamma}$ has the same length as $\gamma$, but it  now satisfies the first case. Thus,  Lemma \ref{l:above}  gives that
\begin{align}
	L \leq c\, [ F(0) - F(\gamma(t_2))]^{\alpha}  \, .
\end{align}

\vskip1mm
\noindent
{\bf{Case 3}}: $F(\gamma(t_1) ) >  F(0) > F(\gamma(t_2) )$.  This follows by introducing a new endpoint $t' \in (t_1 , t_2)$ where $F(\gamma(t')) = F(0)$, 
and using the first case to bound the length of  $\gamma$ on $[t_1 , t']$ and the second case for the interval $[t',t_2]$.  Thus, we get that
\begin{align}
	L \leq c\, [ F(\gamma(t_1)) - F(0)]^{\alpha}  + c\, [ F(0) - F(\gamma(t_2))]^{\alpha}  \, .
\end{align} 
This completes the proof.

\end{proof}

\section{A discrete effective uniqueness theorem}

The Lojasiewicz gradient inequality led to an effective version of the Lojasiewicz uniqueness theorem for 
finite dimensional gradient flows. 
We will need a discrete version of this, where the derivative is replaced by a difference.   This is given in the next proposition which should be thought of as an effective version of lemma $6.9$ in \cite{CM2}.

\begin{Pro}	\label{p:tech}
Given $C \geq 1$ and $\tau \in (1/3 , 1)$, there exist constants $c , \alpha > 0$ so that if $x_j > 0$ is a non-increasing sequence with $x_1 \leq 1$ and
\begin{align}
	x_{j+1}^{1+\tau} \leq C \, (x_j - x_{j+1}) \, , 
\end{align}
then $\sum_j \left| x_j - x_{j+1} \right|^{\frac{1}{2}} \leq c \, x_1^{\alpha}$.
\end{Pro}

\vskip1mm
To illustrate how this will be used, let $\gamma$ be a finite dimensional gradient flow as in the previous section.  If  we set $f= F \circ \gamma$ and
$x_j = f (j) - F(0)$, then the discrete version of \eqr{e:gLhere} says that
\begin{align}
	x_{j+1}^{1+\tau} \leq x_j - x_{j+1} \, .
\end{align}
Thus, Proposition \ref{p:tech} gives that $ \left| x_j - x_{j+1} \right|^{\frac{1}{2}} $ is summable,     bounding the length   by 
\eqr{e:Lgamma}.  We will use a  similar approach in the next section to get an effective uniqueness theorem for mean curvature flow.

\vskip1mm
We will use the following calculus lemma in the proof of Proposition \ref{p:tech}.

\begin{Lem}	\label{l:eleme}
Fix $C \geq 1$ and $\tau \in (1/3 , 1]$.  If $0<b < a \leq 1$ and $b^{1+\tau} \leq C \, (a-b)$, 
then 
\begin{align}	\label{e:eleme}
	b^{-\tau} - a^{-\tau} > \frac{1}{12\, C} \, .
\end{align}
\end{Lem}

\begin{proof}
We will consider two cases.  Suppose first that $2\,b \leq a$, so that 
\begin{align}
	b^{-\tau} - a^{-\tau}  \geq \left( \frac{a}{2} \right) ^{-\tau} - a^{-\tau} = \left(  2^{\tau} - 1 \right) \, a^{-\tau} \geq \left(  2^{\tau} - 1 \right) > 2^{ \frac{1}{3}} - 1 \, .
\end{align}
Since $2^{ \frac{1}{3}} - 1$ is greater than $\frac{1}{12}$ (and $C \geq 1$), we see that \eqr{e:eleme} holds in this case.

In the remaining case, we have that $a < 2 \, b$.  If we define $g(x) = x^{-\tau}$ for $x \in [b,a]$, then $g' = -\tau \, x^{-1-\tau}$ is increasing.  Therefore, the fundamental theorem of calculus gives that
\begin{align}	\label{e:fromFTC}
	a^{-\tau}  - b^{-\tau} = \int_b^a g'(x) \, dx \leq g'(a) \, (a-b) = -\tau  \, a^{-1-\tau} \, (a-b) \, .
\end{align}
Combining this with the assumption that $b^{1+\tau} \leq C \, (a-b)$, we see that
\begin{align}
	b^{-\tau} - a^{-\tau} \geq \tau \, a^{-1-\tau} \, (a-b) \geq \frac{\tau}{C}
		 \, \left( \frac{b}{a} \right)^{1+\tau} > \frac{\tau}{C} \, 2^{-1-\tau} \, ,
\end{align}
where the last inequality used that  $a < 2 \, b$. Since $\tau \in (1/3,1]$, we have that $\tau \, 2^{-1-\tau} > \frac{1}{12}$ and, thus, we get \eqr{e:eleme} in this case as well.
\end{proof}

\begin{proof}[Proof of Proposition \ref{p:tech}]
Lemma \ref{l:eleme} gives for each $j$ that
\begin{align}	\label{e:eleme2}
	x_{j+1}^{-\tau} - x_j^{-\tau} > \frac{1}{12\, C} \, .
\end{align}
Iterating this, we see that
\begin{align}	\label{e:fromiterating}
	x_{j+1}^{-\tau} > x_1^{-\tau} + \frac{j}{12\, C} \, .
\end{align}
Since $\tau \in (1/3,1)$, we have that $\frac{1}{\tau} = 1 + 3 \, \delta$ with $0< \delta < 1$.
The Cauchy-Schwarz inequality (for sums) gives that
\begin{align}	\label{e:inhere}
	\left( \sum_{j=1}^N \left| x_j - x_{j+1} \right|^{\frac{1}{2}} \right)^2 \leq 
	\left( \sum_{j=1}^N ( x_j - x_{j+1}) \, j^{1+\delta} \right) \,  \sum_{j=1}^N  j^{-1-\delta} <
	\frac{2}{\delta} \, \sum_{j=1}^N ( x_j - x_{j+1}) \, j^{1+\delta} \, , 
\end{align}
where the last inequality used that 
\begin{align}
	\sum_{j=1}^{\infty} j^{-1-\delta} < 1 + \int_1^{\infty} x^{-1-\delta}\, dx = 1 + \frac{1}{\delta}= \frac{\delta +1 }{\delta} < \frac{2}{\delta} \, .
\end{align}
Using  the summation by parts formula
 \begin{align}
 	\sum_{j=1}^N b_j \, (x_{j+1} - x_{j}) = \left[  b_{N+1} x_{N+1} - b_1 x_1 \right] -  \sum_{j=1}^N x_{j+1} \, (  b_{j+1}-b_j )  
 \end{align}
 with  $b_j = j^{1+\delta}$, we see that
  \begin{align}
 	\sum_{j=1}^N j^{1+\delta} \, (x_j - x_{j+1} ) \leq  x_1 +   \sum_{j=1}^N x_{j+1} \, ( (j+1)^{1+\delta} -j^{1+\delta})  \leq x_1 + 
	2 \, \sum_{j=1}^N x_{j+1} \,  (j+1)^{\delta} \, , 
 \end{align}
where the last inequality used that  $[ (j+1)^p - j^p] \leq 2 \, (j+1)^{p-1}$ for any $p \in (1,2)$.
Inserting the bound \eqr{e:fromiterating}, this becomes
  \begin{align}	\label{e:midway}
 	\sum_{j=1}^N j^{1+\delta} \, (x_j - x_{j+1} ) \leq    x_1 + 
	2 \, \sum_{j=1}^N   (j+1)^{\delta} \, \left( 
	x_1^{-\tau} + \frac{j}{12\, C}
	\right)^{ -1-3\, \delta} \, .
 \end{align}
  Since $x_1 < 1$, we have that 
  \begin{align}	 
 	 \sum_{j=1}^N   (j+1)^{\delta} \, \left( 
	x_1^{-\tau} + \frac{j}{12\, C}
	\right)^{ -1-3\, \delta} &=  (12\, C)^{\delta} \, \sum_{j=1}^N   \left( \frac{j+1}{12\,C} \right)^{\delta} \, \left( 
	x_1^{-\tau} + \frac{j}{12\, C}
	\right)^{ -1-3\, \delta} \notag \\
	&
	\leq    (12\, C)^{\delta} \, \sum_{j=1}^N    \left( 
	x_1^{-\tau} + \frac{j}{12\, C}
	\right)^{ -1-2\, \delta}   \\
	&
	\leq    (12\, C)^{\delta} \, x_1^{ \tau \, \delta} \,  \sum_{j=1}^N    \left( 
	1 + \frac{j}{12\, C}
	\right)^{ -1-  \delta}
	 \, , \notag
 \end{align}
 where the last inequality used that $x_1 < 1$.
 The last sum is uniformly bounded independent of $N$, 
so we see that
  \begin{align}	 
 	\sum_{j=1}^N j^{1+\delta} \, (x_j - x_{j+1} )   & \leq x_1 + c' \, x_1^{\alpha'} \, , 
 \end{align}
 where $c'=c'(C,\tau)$ and $\alpha' = \alpha' (\tau) > 0$.  Combining this with  \eqr{e:inhere} gives the claim.
\end{proof}

\section{Applications to mean curvature flow}

In this section, we will prove the effective uniqueness theorem for rescaled mean curvature flow and apply it to prove Theorem \ref{t:AB}.
 The   idea is that if a flow is initially close to a cylinder  $\cC$ and its Gaussian density does not change much, then it remains close to $\cC$.
A version of this was used in \cite{CM4} to prove a stratification theorem and sharp estimates for the size of the singular set.   Recall that there is a fixed entropy bound $\lambda_0$ throughout.{\footnote{There are generalizations to the case of a local entropy bound, cf. \cite{S}.}}

\subsection{The Lojasiewicz inequality}

The next theorem from \cite{CM2}   gives a discrete version of a gradient Lojasiewicz-type inequality for  rescaled MCF{\footnote{There are two small differences in the statement here versus in \cite{CM2}.  The first is that \cite{CM2} assumes an entropy bound, while that is assumed everywhere here.  Second, \cite{CM2} was aimed at uniqueness of tangent flows, so there is an implicit assumption there that $F$ is above the cylinder.  That is however not used in the proof there (the first line of the proof is for the absolute values of the difference).}}.  
In the theorem,  $\cC=\SS^k_{\sqrt{2k}} \times \RR^{n-k}$ is a shrinking cylinder.

\begin{Thm}	\label{t:gradMCF}
(Theorem $6.1$ in \cite{CM2}) There exist constants $ C, \bar{R} , \epsilon$ and $\tau \in (1/3,1)$ so that if 
  $\Sigma_s$ is a rescaled MCF for $s \in [t-1, t+1]$  with $\dist_{\bar{R}} ( \Sigma_s , \cC) < \epsilon$ for each $s$, then
\begin{align}	\label{e:thegmcf}
	\left|F(\Sigma_t)-F(\cC) \right|^{1+\tau} &\leq C \, \left(F(\Sigma_{t-1})-F(\Sigma_{t+1})\right) \, .
\end{align}
\end{Thm}

 \vskip1mm
 In the extreme case where $F$ does not change, 
   this theorem  says that  any static solution    that is sufficiently  close to a cylinder on a large set, must have the same gaussian area as the cylinder, cf. \cite{CIM,CM3}.  A similar statement holds in higher codimension, \cite{CM5}; cf. \cite{Z} and see \cite{CS} for asymptotically conical shrinkers.

  \subsection{Staying close to the cylinder}
  
 To keep applying the Lojasiewicz inequality, we must show that the gradient flow does not move very far from the critical point.  This was clear in the model case that directly  bounds the length of the curve (see the discussion after \eqr{e:modelinside}).  This is more subtle now since the natural bound that we get is in $L^1$, but the closeness that is required is as a $C^{2,\alpha}$ graph in a fixed ball.     This is summarized in the following lemma:
  
  \begin{Lem}	\label{l:promote}
  Given $\epsilon > 0$ and $\bar{R} \geq 2n$, there exist $\tilde{C} , \epsilon_1 ,  \mu > 0$ and $ R > \bar{R}$ so that if $\dist_R ( \Sigma_t , \cC)  \leq \epsilon_1$ for $t \in [t_1  , t_1 + 2]$ and 
  \begin{align}	\label{e:andthis}
  	\sum_{j=0}^N \left( F (\Sigma_{t_1+j}) - F (\Sigma_{t_1 + j +1}) \right)^{ \frac{1}{2}} \leq \mu \, , 
  \end{align}
  then $\dist_{\bar{R}} ( \Sigma_{t}, \cC)  \leq \tilde{C} \, \mu  $ for $t \in [t_1 + 1, t_1 + N+1]$
  and $\dist_{\bar{R}} ( \Sigma_{t}, \cC)  \leq   \epsilon$ for $t \in [t_1  , t_1 + N+3]$.
  \end{Lem}
  
  \begin{proof}
  This follows as in $(1)$ on page $268$ of \cite{CM2}.   Since we have an entropy bound and initial closeness to the cylinder, 
  the Brakke estimate, \cite{W1},   gives local a
  priori curvature bounds{\footnote{The Brakke estimate is applied to the
   corresponding MCF, giving an interior estimate forward in time on the MCF. Translating this back to the rescaled MCF, the estimate  now holds on a definite larger scale (because of the scaling of the flow).  This is why there is no loss on the spatial region when the argument is repeated. See page $268$ in \cite{CM2}, cf. \cite{CIM} for the static case.}}.    
  We use the entropy bound, the Cauchy-Schwarz inequality and \eqr{e:andthis}
  to get that
   \begin{align}	\label{e:andthis}
  	\left( 4\, \pi \right)^{ - \frac{n}{2}} \, \sum_{j=0}^N  \int_{t_1+j}^{t_1 + j+1}\,  \int_{\Sigma_t}  |x_t| \, \e^{- \frac{|x|^2}{4}}   & \leq C \, \left( 4\, \pi \right)^{ - \frac{n}{2}} \,  \sum_{j=0}^N \left( \int_{t_1 +j}^{t_1 + j+1}\,  \int_{\Sigma_t}  |x_t|^2 \, \e^{- \frac{|x|^2}{4}}  \right)^{ \frac{1}{2} } \notag \\
	&= C \, 
	\sum_{j=0}^N \left( F (\Sigma_{t_1+j}) - F (\Sigma_{t_1 + j +1}) \right)^{ \frac{1}{2}} < C\, \mu \, , 
  \end{align}
  Combining the  curvature bound  with parabolic estimates and the above $L^1$ bound gives  higher derivatives as desired.
  Repeating this and using the uniform bound from \eqr{e:andthis}
gives the first claim.  The second claim follows from the first.
  \end{proof}
  
  The bound on the distance in the lemma is sufficient for the applications here.  It is possible to obtain finer estimates and structure near a cylinder, cf.
  \cite{CCMS,CHH,G,HK} and \cite{CIKS} for Ricci flow.

\subsection{Proofs of the main results}

 We are now ready to prove the effective uniqueness theorem for rescaled mean curvature flow.  The proof is modeled on the proof of 
 Proposition \ref{p:model} in the model case, with the same three cases depending on the range of $F$.{\footnote{One difference  is that time-reversal is not allowed  for MCF (the regularity is destroyed), so the case where $F$ is below $F(\cC)$ is not equivalent to the   case where it is above.}}

\begin{proof}[Proof of Theorem \ref{t:close}]
Let $C , \bar{R} > 1$,  $\epsilon > 0$ and $ \tau \in (1/3 , 1)$ be given by Theorem \ref{t:gradMCF}.  Then let $R > \bar{R}$ and $ \epsilon_1 , \mu > 0$ be given by
Lemma \ref{l:promote}.
We will consider three cases.  

\vskip1mm
\noindent
{\bf{Case 1}}: $F(\Sigma_{t_2} ) \geq F (\cC)$.  By continuity and a limiting argument, we can assume that 
$F(\Sigma_{t_2} ) > F (\cC)$. Define a sequence $x_j > 0$ by
\begin{align}
	x_j = F(\Sigma_{t_1 + 2j -1} ) - F (\cC) \, .
\end{align}
This sequence is non-increasing since $\Sigma_t$ is the gradient flow for $F$.  Define   $N$ to be the largest integer with 
\begin{align}
	 \sum_{j=1}^n |x_j - x_{j+1}|^{\frac{1}{2}} \leq \mu \, .
\end{align}
It follows from Lemma \ref{l:promote} and Theorem \ref{t:gradMCF} (and the monotonicity of $F$) that for $j \leq N$
\begin{align}	\label{e:herett}
	x_{j+1}^{1+\tau} \leq C \, (x_j - x_{j+1}) \, .
\end{align}
Therefore, Proposition \ref{p:tech}
 gives $c , \alpha > 0$ so that  
\begin{align}	\label{e:gotcha}
	\sum_{j\leq N}\,  \left| x_j - x_{j+1} \right|^{\frac{1}{2}} \leq c \, x_1^{\alpha} \, .
\end{align}
As long as $x_1$ is small enough, this is below $\mu$ and we see that $N$ takes us all the way to $t_2$.  Combining \eqr{e:gotcha} with 
  Lemma \ref{l:promote}  gives the claim.

 \vskip1mm
 \noindent
{\bf{Case 2}}: $F(\cC) \geq F(\Sigma_{t_1}) $.  By continuity and a limiting argument, we can assume that 
$F(\cC) > F(\Sigma_{t_2} ) $.   Define a sequence $y_j < 0$ by
\begin{align}
	y_j = F(\Sigma_{t_1 + 2j -1} ) - F (\cC) \, .
\end{align}
Note that the negative sign on the $y_j$'s means that $|F(\Sigma_t) - F(\cC)|$ is non-decreasing, which is the opposite of what we had in case 1.
Define   $N$ to be the largest integer with 
\begin{align}
	 \sum_{j=1}^n |y_j - y_{j+1}|^{\frac{1}{2}} \leq \mu \, .
\end{align}
It follows from Lemma \ref{l:promote} and Theorem \ref{t:gradMCF}   that for $j \leq N$
\begin{align}	\label{e:herett2}
	|y_{j}|^{1+\tau} \leq   C \, (y_j - y_{j+1}) \, ,
\end{align}
where the negative sign on $y_j$ is why we bound $|y_{j}|^{1+\tau}$ instead of $|y_{j+1}|^{1+\tau}$.  

We now want to work backwards, so we set $x_j = -y_{N-j}$.  This makes the $x_j$'s positive, non-increasing, and gives us \eqr{e:herett} (in place of \eqr{e:herett2}).
As in case 1,  Proposition \ref{p:tech}
 gives $c , \alpha > 0$ so that  
\begin{align}	\label{e:gotcha2}
	\sum_{j\leq N}\,  \left| x_j - x_{j+1} \right|^{\frac{1}{2}} \leq c \, x_1^{\alpha} \, .
\end{align}
As long as $x_1$ is small enough, this is below $\mu$ and we see that $N$ takes us all the way to $t_2$.  Combining \eqr{e:gotcha2} with 
  Lemma \ref{l:promote}  gives the claim.

 \vskip1mm
 \noindent
{\bf{Case 3}}: $  F(\Sigma_{t_1} ) > F(\cC) >  F(\Sigma_{t_2}) $.    This time we divide the sequence $F(\Sigma_{t_1 + 2j -1} ) - F (\cC) $ into two sequences (the positive ones, then the negative ones) and argue as in the two previous cases on each part.

\end{proof}

\subsection{Applications}

A MCF $M_s$ gives a
 rescaled MCF   $\Sigma_t$ is 
 given  by setting $\Sigma_t=\frac{1}{\sqrt{-s}}M_s$,
$t=-\log(-s)$, $s<0$.

Using the theorem, we can now prove Theorem \ref{t:AB}:

\begin{proof}[Proof of Theorem \ref{t:AB}]
For each $i$, define a rescaled MCF $\Sigma^i$ centered at $(x_i , t_i)$ by
\begin{align}
	\Sigma^i_t = \e^{ \frac{t}{2}} \, \left( M_{s_i - \e^{-t}} - x_i \right) \, .
\end{align}
Set $T_i= - 2 \, \log \mu_i$, so that $T_i \to \infty$ since $\mu_i \to 0$.

By the assumption (B), we have that $\Sigma^i_{T_i}  = \frac{1}{\mu_i} \,  \left( M_{s_i - \mu_i^2} - x_i \right)$ satisfies
\begin{align}
	\Sigma^i_{T_i}   &\to \cO(\cC) \, , \\
	F( \Sigma^i_{T_i}) &\to F(\cC) \, .
\end{align}
Furthermore, since $x_i \to 0$ and $s_i \to 0$ and  the origin is a cylindrical singularity, given any $\delta_1 , R > 0$,  we can choose $i'$ and $t_0 > 0$ so that
\begin{align}
 \dist_R (\Sigma^i_{t} , \cC ) &<\delta_1 \, , \\
 \left| F( (\Sigma^i_{t} ) - F(\cC) \right| &< \delta_1 
 \end{align}
   for $t \in [t_0 - 1, t_0+2]$ for all $i \geq i'$.
Thus, we can apply Theorem \ref{t:close} to get that the distance from $\cC$ to $\cO(\cC)$ is arbitrarily small, completing the proof.
\end{proof}

  It is well-known   that the uniqueness of cylindrical blow downs is a consequence of  \cite{CM2}.  It also 
  follows easily from Theorem \ref{t:close}. 
  
When $M_s$ is an ancient MCF with bounded entropy,  a blow down is a limit of a sequence of rescaled flows $M_s^i = \lambda_i^{-1} \, M_{\lambda_i^2 \, s}$ with $\lambda_i \to \infty$.
It follows from Huisken's monotonicity \cite{H}, that the blow down must be self-similarly shrinking.  If one blow down (i.e., a limit for some sequence $\lambda_i \to \infty$) is cylindrical, then we can construct rescaled MCF's as in the proof of Theorem \ref{t:AB} that start close to a cylinder and so that the $F$ functional stays close to $F(\cC)$.  Theorem \ref{t:close} then gives that this rescaled flow is almost a fixed point on the entire stretch and, thus, every blow down is the same cylinder.

\end{document}